\documentclass[12pt,a4paper]{amsart}
\usepackage{amsfonts,color}
\usepackage{amsthm}
\usepackage{amsmath}
\usepackage{amscd}
\usepackage{amssymb}
\usepackage[latin2]{inputenc}
\usepackage{t1enc}
\usepackage[mathscr]{eucal}
\usepackage{indentfirst}
\usepackage{graphicx}
\usepackage{graphics}
\usepackage{pict2e}
\usepackage{epic}
\usepackage{url}
\usepackage{epstopdf}
\usepackage{comment}
\usepackage{tikz}

\numberwithin{equation}{section}
\usepackage[margin=2.6cm]{geometry}

\allowdisplaybreaks

\theoremstyle{plain}
\newtheorem{Th}{Theorem}[section]
\newtheorem{Lemma}[Th]{Lemma}

 \theoremstyle{definition}

\newtheorem{?}[Th]{Problem}

\begin{document}

\title{The Merino--Welsh conjecture is false for matroids}

\author{Csongor Beke}

\address{Trinity College, University of Cambridge, CB2 1TQ, United Kingdom \\
}

\email{bekecsongor@gmail.com}

\author{Gergely K\'al Cs\'aji}

\address{ Institute of Economics, Centre for Economic and Regional Studies, Hungary, H-1097 Budapest, T\'{o}th K\'{a}lm\'{a}n u. 4\\ 
}

\email{csaji.gergely@krtk.hun-ren.hu}

\author[P. Csikv\'ari]{P\'{e}ter Csikv\'{a}ri}

\address{HUN-REN Alfr\'ed R\'enyi Institute of Mathematics, H-1053 Budapest Re\'altanoda utca 13-15 \and ELTE: E\"{o}tv\"{o}s Lor\'{a}nd University \\ Mathematics Institute, Department of Computer
Science \\ H-1117 Budapest
\\ P\'{a}zm\'{a}ny P\'{e}ter s\'{e}t\'{a}ny 1/C}

\email{peter.csikvari@gmail.com}

\author{S\'ara Pituk}

\address{ELTE: E\"{o}tv\"{o}s Lor\'{a}nd University \\ H-1117 Budapest
\\ P\'{a}zm\'{a}ny P\'{e}ter s\'{e}t\'{a}ny 1/C}

\email{pituksari@gmail.com}

\thanks{P\'eter Csikv\'ari is supported by the MTA-R\'enyi Counting in  Sparse Graphs ``Momentum'' Research Group and by the Dynasnet ERC Synergy
project (ERC-2018-SYG 810115). Gergely K\'al Cs\'aji is supported by the Hungarian Scientific Research Fund, OTKA, Grant No. K143858, by the Momentum Grant of the Hungarian Academy of Sciences, grant number 2021-1/2021 and by the Ministry of Culture and Innovation of Hungary from the National Research, Development and Innovation fund, financed under the KDP-2023 funding scheme (grant number C2258525). S\'ara Pituk is supported by the Ministry of Culture and Innovation and the National Research, Development and Innovation Office within the Quantum Information National Laboratory of Hungary (Grant No. 2022-2.1.1-NL-2022-00004)}



\begin{abstract} The matroidal version of the Merino--Welsh conjecture states that the Tutte polynomial $T_M(x,y)$ of any matroid $M$ without loops and coloops satisfies that
$$\max(T_M(2,0),T_M(0,2))\geqslant T_M(1,1).$$
Equivalently, if the Merino--Welsh conjecture is true for all matroids without loops and coloops, then the following inequalities are also satisfied for all matroids without loops and coloops:
$$T_M(2,0)+T_M(0,2)\geqslant 2T_M(1,1),$$
and 
$$T_M(2,0)T_M(0,2)\geqslant T_M(1,1)^2.$$
We show a counter-example for these inequalities.
\end{abstract}

\maketitle

\section{Introduction}
For a connected graph $G$, let $\tau(G)$, $\alpha(G)$ and $\alpha^*(G)$ denote the number of spanning trees, the number of acyclic orientations and the number of strongly connected orientations, respectively. Merino and Welsh \cite{merino1999forests} conjectured that if $G$ is a connected graph without loops and bridges, then 
$$\max(\alpha(G),\alpha^*(G))\geqslant \tau(G).$$
Note that $\alpha(G),\alpha^*(G)$, and $\tau(G)$ are all evaluations of the Tutte polynomial, namely, $T_G(2,0)=\alpha(G)$, $T_G(0,2)=\alpha^*(G)$, and $T_G(1,1)=\tau(G)$, where the Tutte polynomial $T_G(x,y)$ is defined as
$$T_G(x,y)=\sum_{A\subseteq E}(x-1)^{k(A)-k(E)}(y-1)^{k(A)+|A|-v(G)},$$
with $k(A)$ denoting the number of connected components of the graph $(V,A)$, see \cite{tutte1954contribution}.
There is a vast amount of literature on the properties of the Tutte polynomial and its applications, for instance, \cite{brylawski1992tutte,crapo1969tutte,ellis2011graph,welsh1999tutte}, or the book \cite{ellis2022handbook}.

Conde and Merino \cite{conde2009comparing} also suggested the following ``additive'' and ``multiplicative'' versions of the conjecture:
$$T_G(2,0)+T_G(0,2)\geqslant 2T_G(1,1),$$
and 
$$T_G(2,0)T_G(0,2)\geqslant T_G(1,1)^2.$$
It is easy to see that the multiplicative version implies the additive version which in turn implies the maximum version.

The Merino--Welsh conjecture and its variants triggered considerable attention. Thomassen \cite{thomassen2010spanning} proved that the conjecture is true if the graph $G$ is sufficiently sparse or sufficiently dense. Lin \cite{lin2013note} proved it for $3$-connected graphs satisfying certain degree conditions. Noble and Royle \cite{noble2014merino} proved the multiplicative version for series-parallel graphs. 

The Tutte polynomial naturally extends to matroids. Recall that a matroid $M$ is a pair $(E,\mathcal{I})$ such that $\mathcal{I}\subseteq 2^{E}$, called the independent sets,  satisfying the axioms 
(i) $\emptyset \in \mathcal{I}$, (ii) if $A'\subseteq A\in \mathcal{I}$, then $A'\in \mathcal{I}$, and (iii) if $A,B\in \mathcal{I}$ such that $|B|<|A|$, then there exists an $x\in A\setminus B$ such that $B\cup \{x\} \in \mathcal{I}$. Given a set $S\subseteq E$, the maximal independent subsets of $S$ all have the same cardinality, and this cardinality is called the rank of the matroid, denoted by $r(S)$. The maximum size independent sets of $M$ are called bases, and their set is denoted by $\mathcal{B}(M)$. The dual of a matroid $M$ is the matroid  $M^*$ whose bases are $\{E\setminus B\ |\ B\in \mathcal{B}(M) \}$. For further details on matroids, see for instance \cite{oxley1992matroid}

Given a graph $G=(V,E)$, the edge sets of the spanning forests of $G$ form the independent sets of a matroid $M_G$ called the cycle matroid of $G$. If $G$ is connected, then the basis of $M_G$ are the spanning trees of $G$. One can define the Tutte polynomial of a matroid 
 as $$T_M(x,y)=\sum_{S\subseteq E}(x-1)^{r(E)-r(S)}(y-1)^{|S|-r(S)},$$
where $r(S)$ is the rank of a set $S\subseteq E$. When $M=M_G$, then $T_{M_G}(x,y)=T_G(x,y)$. A loop in a matroid $M$ is an element $x\in E$ such that $r(\{x\})=0$, that is, $\{x\}\notin \mathcal{I}$, and a coloop is an element that is a loop in the dual $M^*$ of the matroid $M$. Equivalently, a coloop is an element that is in every base of $M$. For a cycle matroid $M_G$, loops correspond to loop edges and coloops correspond to bridges in the graph $G$.

Hence it was suggested that the inequalities  
$$\max(T_M(2,0),T_M(0,2)\geqslant T_M(1,1),$$
$$T_M(2,0)+T_M(0,2)\geqslant 2T_M(1,1),$$
$$T_M(2,0)T_M(0,2)\geqslant T_M(1,1)^2$$
may hold true for all matroids $M$ without loops and coloops. (These versions appear explicitly in \cite{ferroni2023merino}, but were treated much earlier without explicitly calling them conjectures.) 
 Note that for general matroids, all these versions are equivalent in the following sense: if one of them is true for all matroids, then the others are also true for all matroids. Applying the maximum version to $M\oplus M^*$ with $M^*$ being the dual of $M$ leads to the multiplicative version of the conjecture. (Here $M\oplus N$ denotes the disjoint union of the matroids $M$ and $N$.)

Knauer, Mart\'inez-Sandoval, and  Ram\'irez Alfons\'in \cite{knauer2018tutte} proved that the class of lattice path matroids satisfies the multiplicative version. Iba{\~n}ez, Merino and Rodr\'iguez \cite{merino2009note} proved the maximum version for some families of graphs and matroids. Ch\'avez-Lomel\'i, Merino, Noble and Ram\'irez-Ib\'a{\~n}ez \cite{chavez2011some} proved the additive version for paving matroids without coloops. In fact, they showed that the polynomial $T_M(x,2-x)$ is convex on the interval $[0,2]$ for these matroids. 
Recently, Ferroni and Schr\"oter \cite{ferroni2023merino} proved the multiplicative version of the conjecture for split matroids.
 Kung \cite{kung2021inconsequential} proved the additive version for some special matroids based on their size and rank. Jackson \cite{jackson2010inequality} proved that
$$T_M(3,0)T_M(0,3)\geqslant T_M(1,1)^2$$
for matroids without loops and coloops.
He phrased it for graphs but he also noted that his proof extends to matroids.

The aim of this short note is to give a counter-example for these inequalities for general matroids.

\begin{Th} \label{counter example}
There are infinitely many matroids $M$ without loops and coloops for which 
$$T_M(2,0)T_M(0,2)<T_M(1,1)^2.$$
\end{Th}

In fact, we show the following slightly stronger result. Let $x_0$ be the largest root of the polynomial $x^3-9(x-1)$. We have $x_0\approx 2.2268...$

\begin{Th} \label{threshold} If $0\leqslant x<x_0$, then there are infinitely many matroids $M$ without loops and coloops for which 
$$T_M(x,0)T_M(0,x)<T_M(1,1)^2.$$
\end{Th}

It is interesting to compare this result with the above inequality of Jackson.
In the paper \cite{beke2023permutation}, the authors of this paper show that $3$ can be improved to $2.9243$.
\bigskip

\textbf{Organization of the paper.}
In the next section, we prove Theorems~\ref{counter example} and \ref{threshold}. Then we give some insight into where the counter-example came from. We end the paper with some concluding remarks.

\section{Counter-examples}

The counter-example for the multiplicative version of the Merino--Welsh conjecture is surprisingly simple. Let $U_{n,r}$ be the uniform matroid on $n$ elements with rank $r$. Let $U^{(2)}_{n,r}$ be the $2$-thickening of $U_{n,r}$, that is, we replace each element of $U_{n,r}$ with $2$ parallel elements. We will show that if $x<x_0$, then $M_n=U^{(2)}_{n,\frac{2}{3}n}$ satisfies the theorem for large enough $n$ if $n$ is divisible by $3$, hence concluding Theorems~\ref{threshold} and \ref{counter example}.

The computation of the Tutte polynomial of $U^{(2)}_{n,r}$ relies on two well-known lemmas.

\begin{Lemma}[Formula (2.24) in \cite{merino2012tutte}]
The Tutte polynomial of the matroid $U_{n,r}$ is the following:
$$T_{U_{n,r}}(x,y)=\sum_{i=1}^r\binom{n-i-1}{n-r-1}x^i+\sum_{j=1}^{n-r}\binom{n-j-1}{r-1}y^j$$
if $0<r<n$, and $T_{U_{n,n}}(x,y)=x^n$ and $T_{U_{n,0}}(x,y)=y^n$.
\end{Lemma}

\begin{Lemma}[Jaeger, Vertigan and Welsh \cite{jaeger1990computational}, formula (3.47) of \cite{merino2012tutte}] \label{thick} {\ \\}
Let $M$ be a matroid, and let $M^{(k)}$ be its $k$-thickening, that is, we replace each element of $M$ with $k$ parallel elements. Then
$$T_{M^{(k)}}(x,y)=(y^{k-1}+y^{k-2}+\dots +1)^{r(M)}T_M\left(\frac{y^{k-1}+y^{k-2}+\dots +y+x}{y^{k-1}+y^{k-2}+\dots +y+1},y^k\right).$$
\end{Lemma}

\bigskip

\noindent By Lemma~\ref{thick} we have
$$T_{M^{(2)}}(x,0)=T_M(x,0),$$
$$T_{M^{(2)}}(0,x)=(x+1)^{r(M)}T_M\left(\frac{x}{x+1},x^2\right),$$
and
$$T_{M^{(2)}}(1,1)=2^{r(M)}T_M(1,1).$$
Clearly, these expressions together with the exact formula for $T_{U_{n,r}}(x,y)$ make the computation of $T_{U^{(2)}_{n,r}}(x,y)$ very fast for specific values of $n,r,x,y$. 

The matroid with the smallest number of elements that we are aware of being a counter-example to the multiplicative version of the Merino--Welsh conjecture is $M=U^{(2)}_{33,22}$ with $66$ elements. For this matroid, we have
$T_{M}(2,0)=8374746166$, 
$T_{M}(0,2)=64127582356390782814$, 
$T_{M}(1,1)=811751838842880$,
and
$$\frac{T_{M}(2,0)T_{M}(0,2)}{T_{M}(1,1)^2}\approx 0.815...$$
\bigskip

To prove Theorem~\ref{threshold}, our next goal is to understand the exponential growth of $T_{U_{n,r}}(x,0)$.

\begin{Lemma} Let $r=n\alpha$ and $x>1$, then
$$T_{U_{n,r}}(x,0)=\begin{cases} f(n) \exp(nH(\alpha)) & \mathrm{if}\  x<\frac{1}{\alpha},\\
f(n)(x(x-1)^{\alpha-1})^n & \mathrm{if}\  x\geqslant \frac{1}{\alpha},
\end{cases}$$
where $n^K>f(n)>n^{-K}$ for some fixed $K$, and $H(\alpha)=-\alpha \ln(\alpha)-(1-\alpha)\ln(1-\alpha)$.
\end{Lemma}

\begin{proof}
We can determine the dominating term of $T_{U_{n,r}}(x,0)$ by comparing two neighboring terms:
$$\binom{n-i-1}{n-r-1}x^i\geqslant \binom{n-i-2}{n-r-1}x^{i+1}\ \ \ 
\text{if and only if}\ \ \ 
\frac{n-i-1}{r-i}\geqslant x.$$
Hence, $\binom{n-i-1}{n-r-1}x^i$ is maximized at
$\left\lceil \frac{xr-(n-1)}{x-1}\right\rceil$. If the right-hand side is negative, then the dominating term is at $i=1$ and $\binom{n-2}{n-r-1}\sim \binom{n}{n-r}\sim \exp(nH(\alpha))$, where $\sim$ means the estimation is valid up to some $n^{K}$. When $x=\frac{1}{\alpha}$, then $\exp(H(\alpha))=x(x-1)^{\alpha-1}$, so we can assume that $x\geqslant \frac{1}{\alpha}$ in the rest of the proof since then on the whole interval $\left(1,\frac{1}{\alpha}
\right)$ we have $T_{U_{n,r}}(x,0)\sim \exp(nH(\alpha))$.

For the sake of simplicity, we carry out the estimation of the dominating term at $$i=\frac{xr-n}{x-1}=\frac{x\alpha-1}{x-1}n$$ 
and we drop the integer part. All these changes affect our computation up to a term $n^{-K}$. In the forthcoming computation, we also estimate $m!\sim \left(\frac{m}{e}\right)^m$ as the terms $\sqrt{2\pi m}(1+o(1))$ can be integrated into $f(n)$:
\begin{align*}
\binom{n-i-1}{n-r-1}x^i&\sim \binom{n-i}{n-r}x^i\\
&\sim \frac{\left(\frac{n-i}{e}\right)^{n-i}}{\left(\frac{n-r}{e}\right)^{n-r}\left(\frac{r-i}{e}\right)^{r-i}}x^i\\
&=\frac{\left(n\left(1-\frac{x\alpha-1}{x-1}\right)\right)^{n\left(1-\frac{x\alpha-1}{x-1}\right)}}{(n(1-\alpha))^{n(1-\alpha)}\left(n\left(\alpha-\frac{x\alpha-1}{x-1}\right)\right)^{n\left(\alpha-\frac{x\alpha-1}{x-1}\right)}}x^i\\
&= \left(\frac{\left(\frac{x(1-\alpha)}{x-1}\right)^{\frac{x(1-\alpha)}{x-1}}}{(1-\alpha)^{1-\alpha}\left(\frac{1-\alpha}{x-1}\right)^{\frac{1-\alpha}{x-1}}}\right)^nx^i\\
&= \left(x^{\frac{x(1-\alpha)}{x-1}}(x-1)^{\alpha-1}\right)^nx^i\\
&= \left(x^{\frac{x(1-\alpha)}{x-1}}(x-1)^{\alpha-1}\right)^nx^{\frac{x\alpha-1}{x-1}n}\\
&=(x(x-1)^{\alpha-1})^n,
\end{align*}
and the result follows.

\end{proof}

\begin{Lemma}
Let $r=\alpha n$ and assume that $x\geqslant \frac{1}{\alpha}$ and $x^2\geqslant \frac{1}{1-\alpha}$. Then for the matroid $M=U^{(2)}_{n,r}$, we have
$$\frac{T_M(1,1)^2}{T_M(x,0)T_M(0,x)}= g(n)\left(\frac{2^{2\alpha}}{\alpha^{2\alpha}{(1-\alpha)^{2(1-\alpha)}}}\cdot \frac{x-1}{x^3}\right)^n,$$
where $n^K>g(n)>n^{-K}$ for some fixed $K$.
\end{Lemma}

\begin{proof}
We have 
$$T_M(1,1)=2^rT_{U_{n,r}}(1,1))=2^{r}\binom{n}{r}\sim \left(\frac{2^{\alpha}}{\alpha^{\alpha}{(1-\alpha)^{1-\alpha}}}\right)^n.$$
Furthermore, 
$$T_M(x,0)=T_{U_{n,r}}(x,0)\sim (x(x-1)^{\alpha-1})^n$$
as $x\ge\frac{1}{\alpha}$.
Finally,
$$T_M(0,x)=(x+1)^rT_{U_{n,r}}\left(\frac{x}{x+1},x^2\right)=(x+1)^rT_{U_{n,r}}\left(\frac{x}{x+1},0\right)+(x+1)^rT_{U_{n,r}}\left(0,x^2\right).$$
Here, the second term will dominate the first one as $T_{U_{n,r}}\left(\frac{x}{x+1},0\right)<T_{U_{n,r}}(1,1)\sim\exp(nH(\alpha))$, while
$$T_{U_{n,r}}\left(0,x^2\right)=T_{U_{n,n-r}}(x^2,0)\sim (x^2(x^2-1)^{(1-\alpha)-1})^n$$
as $x^2\geqslant \frac{1}{1-\alpha}$. Putting everything together, we get that
\begin{align*}\frac{T_M(1,1)^2}{T_M(x,0)T_M(0,x)}&\sim \left(\frac{2^{2\alpha}}{\alpha^{2\alpha}{(1-\alpha)^{2(1-\alpha)}}}\right)^n(x(x-1)^{\alpha-1}(x+1)^{\alpha}x^2(x^2-1)^{-\alpha})^{-n}\\
&\sim \left(\frac{2^{2\alpha}}{\alpha^{2\alpha}{(1-\alpha)^{2(1-\alpha)}}}\cdot \frac{x-1}{x^3}\right)^{n}.
\end{align*}
\end{proof}

\begin{proof}[Proof of Theorem~\ref{threshold}]
The maximum of the function $\frac{2^{2\alpha}}{\alpha^{2\alpha}{(1-\alpha)^{2(1-\alpha)}}}$ is at $\alpha=\frac{2}{3}$, where it takes value $9$.
We can assume by monotonicity that $2\leqslant x<x_0$. Then $x\geqslant \frac{1}{\alpha}=\frac{3}{2}$ and $x^2\geqslant \frac{1}{1-\alpha}=3$, whence for $M=U^{(2)}_{n,\frac{2}{3}n}$,we get that
$$\frac{T_M(1,1)^2}{T_M(x,0)T_M(0,x)}\geqslant n^{-K}\left(\frac{9(x-1)}{x^3}\right)^n>1$$
for large enough $n$ as $\frac{9(x-1)}{x^3}>1$.
\end{proof}

\section{Intuition behind the counter-examples}

In this section, we try to explain the underlying intuition behind the counter-examples. 

It turns out that the Merino--Welsh conjecture is strongly related to the ``local structure'' of a matroid. To make this statement more precise, we need the concept of the local basis exchange graph. This is a bipartite graph associated with a basis $B\in \mathcal{B}(M)$ of the matroid $M$ whose parts are the elements of $B$ on one side, and the non-elements on the other side. We connect an element $b\in B$ with $c\in E\setminus B$ if $(B\setminus \{b\})\cup \{c\}$ is again a basis. Let us call this bipartite graph $H_M[B]$. It turns out \cite{beke2023permutation} that one can associate a polynomial $\widetilde{T}_{H_M[B]}(x,y)$ to each local basis exchange graph such that 
$$T_M(x,y)=\sum_{B\in \mathcal{B}}\widetilde{T}_{H_M[B]}(x,y).$$
We call the polynomial $\widetilde{T}_{H}(x,y)$ the permutation Tutte polynomial of the graph $H$ in the paper \cite{beke2023permutation}. We do not go into details about the actual definition of this polynomial as we only need one key observation about it: if for all $B\in \mathcal{B}(M)$ we have 
$$\widetilde{T}_{H_M[B]}(2,0)\widetilde{T}_{H_M[B]}(0,2)\geqslant \widetilde{T}_{H_M[B]}(1,1)^2,$$
then
$$T_M(2,0)T_M(0,2)\geqslant T_M(1,1)^2.$$
For the balanced complete bipartite graph $K_{n,n}$, we have
$$\frac{\widetilde{T}_{K_{n,n}}(2,0)\widetilde{T}_{K_{n,n}}(0,2)}{\widetilde{T}_{K_{n,n}}(1,1)^2}\approx \frac{n\pi}{4},$$
which shows that the required inequality holds, but the ratio is not exponential in $n$. Most probably, the same is true for every sufficiently dense balanced bipartite graph. 

It turns out that gluing pendant edges to one side of a bipartite graph may actually decrease the studied ratio. So it is natural to study bipartite graphs that are obtained from balanced complete bipartite graphs by gluing a pendant edge to each vertex on one side. For sufficiently large $n$ these bipartite graphs indeed violate the inequality $\widetilde{T}_{H}(2,0)\widetilde{T}_{H}(0,2)\geqslant \widetilde{T}_{H}(1,1)^2$. 

The next question is whether we can construct a matroid $M$ with the desired local basis exchange graphs. For the matroid $U_{n,r}$, the local basis exchange graph is $K_{r,n-r}$. When we apply the $2$-thickening transformation to any matroid $M$, then we obtain $H_{M^{(2)}}[B]$ from $H_M[B]$ as follows: each vertex $b\in B$ gets a pendant edge, and each $c\in E\setminus B$ gets a twin $c'$, i.e. a vertex that is connected to the same vertices as $c$. So in this way $K_{r,n-r}$ is transformed into $K_{r,2(n-r)}$ with a pendant edge attached to each of the $r$ basis elements. If we choose $r=\frac{2}{3}n$, then we get a balanced complete bipartite graph with a pendant edge attached to each vertex on one side of the graph. This is exactly the bipartite graph that we needed, and we get this graph as the local basis exchange graph associated with every basis.

Let us mention that there are graphical matroids where some of the local basis exchange graphs are isomorphic to the above bipartite graphs, but unfortunately, not all of them. 

We believe that the theory of the permutation Tutte polynomial developed in the paper \cite{beke2023permutation} can attack successfully variations of the Merino--Welsh conjecture, or help identify critical structures.

\section{Concluding remarks}

Ferroni and Schr\"oter \cite{ferroni2022valuative} lists five major open problems concerning invariants of matroids on the fourth page. The fifth one is the matroidal version of the Merino--Welsh conjecture. Interestingly, another conjecture of this list has been disproved too, namely, the Ehrhart positivity of matroids. For this conjecture, Ferroni provided counter-examples \cite{ferroni2022matroids}. This shows that one should be careful with these conjectures, matroids are much more versatile than we may expect them, and small examples may be misleading. The program initiated by Ferroni and Schr\"oter \cite{ferroni2022valuative} generating many examples may lead to more counter-examples for various conjectures.

Though the counter-examples given in this note are not graphical matroids, they still advise caution concerning the Merino--Welsh conjecture. It is also important to note that the three forms of the Merino--Welsh conjecture are not equivalent for graphs, and it may occur, for instance, that the original version is true, while the multiplicative version is false.

\bibliography{references}
\bibliographystyle{plain}

\end{document}